\documentclass[microtype]{gtpart}     

\usepackage{amsmath,amsthm,amsfonts,latexsym,amscd,amssymb,enumerate}
\usepackage[all]{xy}
\usepackage[mathscr]{eucal}

\usepackage{amsrefs}

\swapnumbers
\theoremstyle{plain}
\newtheorem{theorem}{Theorem}[section]
\newtheorem{lemma}[theorem]{Lemma}
\newtheorem{corollary}[theorem]{Corollary}
\newtheorem{proposition}[theorem]{Proposition}

\theoremstyle{definition}
\newtheorem{definition}[theorem]{Definition}

\theoremstyle{remark}
\newtheorem{remark}[theorem]{Remark}

\newcommand{\bB}{\mathbb{B}}

\newcommand{\bK}{\mathbb{K}}

\newcommand{\bN}{\mathbb{N}}

\newcommand{\bZ}{\mathbb{Z}}

\newcommand{\cB}{\mathcal{B}}

\newcommand{\cH}{\mathcal{H}}

\newcommand{\cK}{\mathcal{K}}

\newcommand{\cQ}{\mathcal{Q}}

\newcommand{\cU}{\mathcal{U}}
\newcommand{\cV}{\mathcal{V}}
\newcommand{\cW}{\mathcal{W}}
\newcommand{\cX}{\mathcal{X}}

\newcommand{\rmK}{\mathrm{K}}

\newcommand{\sH}{\mathscr{H}}

\newcommand{\sK}{\mathscr{K}}

\newcommand{\Nu}{\mathcal{V}}

\newcommand{\complexs}{\mathbb{C}}

\newcommand{\integers}{\mathbb{Z}}

\DeclareMathOperator{\id}{id}
\newcommand{\boundary}[1]{\partial#1}

\newcommand{\tensor}{\otimes}
\newcommand{\into}{\hookrightarrow}
\newcommand{\onto}{\twoheadrightarrow}
\newcommand{\iso}{\cong}
\newcommand{\disjointunion}{\amalg}

\DeclareMathOperator{\pr}{pr}

\DeclareMathOperator{\ind}{ind}
\DeclareMathOperator{\sgn}{Sgn}

\DeclareMathOperator{\Ind}{Ind}

\newcommand{\forget}[1]{}

\newcommand{\innerprod}[1]{\langle #1 \rangle}

{\catcode`@=11\global\let\c@equation=\c@theorem}
\renewcommand{\theequation}{\thetheorem}

\renewcommand{\theenumi}{(\arabic{enumi})}

\newcounter{commentcounter}

\usepackage{ifthen,srcltx,color}
\newcommand{\showcomments}{yes}

\newsavebox{\commentbox}
\newenvironment{com}%
{\ifthenelse{\equal{\showcomments}{yes}}%
{\raisebox{0.7mm}{\small\cref{com:\arabic{commentcounter}}}
        \begin{lrbox}{\commentbox}
        \begin{minipage}[t]{1.2in}\raggedright\sffamily\tiny
        \raisebox{0.7mm}{\thecommentcounter}\label{com:\arabic{commentcounter}}}
{\begin{lrbox}{\commentbox}}}
{\ifthenelse{\equal{\showcomments}{yes}}
{\end{minipage}\end{lrbox}\marginpar{\usebox{\commentbox}}}
{\end{lrbox}}}

\usepackage[capitalise]{cleveref}
\crefname{setup}{Setup}{Setups}
\crefname{lemma}{Lemma}{Lemmas}
\crefname{diagram}{Diagram}{Diagram}
\crefname{commentcounter}{}{}
\crefdefaultlabelformat{#2\textup{#1}#3}

\usepackage{xparse}

\title{The Gromov-Lawson codimension $2$ obstruction to positive scalar
  curvature and the $C^*$-index}

\author{ Yosuke Kubota}
\givenname{Yosuke}
\surname{Kubota}
\address{Department of Mathematical Sciences \\ Shinshu University \\
Japan }
\email{ykubota@shinshu-u.ac.jp}

\author{Thomas Schick}
\givenname{Thomas}
\surname{Schick}
\address{Mathematisches Institut\\
Universit\"at G{\"o}ttingen\\
Germany}
\email{thomas.schick@math.uni-goettingen.de}
\urladdr{http://www.uni-math.gwdg.de/schick}


\arxivreference{1909.09584}

\begin{document}
\pagestyle{myheadings}
\markboth{Yosuke Kubota, Thomas Schick}{On the Gromov-Lawson codimension $2$
  obstruction to positive scalar curvature}

\begin{abstract}
Gromov and Lawson developed a codimension $2$ index obstruction to positive
  scalar curvature for a closed spin manifold $M$, later refined by Hanke, Pape and Schick. Kubota has shown that also this obstruction can be obtained from the Rosenberg index of the ambient manifold $M$ which takes values in the K-theory of the maximal $C^*$-algebra of the fundamental group of $M$, using relative index constructions. 
  
In this note, we give a slightly simplified account of Kubota's work and
remark that it also applies to the signature operator, thus recovering the
homotopy invariance of higher signatures of codimension $2$ submanifolds of
Higson, Schick, Xie.

\end{abstract}

\maketitle

\section{Introduction}

Let $M$ be a closed spin manifold and $N\subset M$ a submanifold of
codimension $2$ with trivial normal bundle. Much current research is devoted
to the question when such a manifold $M$ does admit a Riemannian metric of
positive scalar curvature.
The submanifold $N$ can provide an obstruction to this, as was first explored
in rather special situations by Gromov and Lawson \cite{GromovLawson}. The
core of the argument lead to the following version \cite{HankePapeSchick}:
\begin{theorem}\label{theo:cd2o}
  Let $M$ be a closed manifold, $N\subset M$ a connected submanifold
  codimension $2$ 
  with trivial normal bundle. Assume that the induced map $\pi_1(N)\to
  \pi_1(M)$ is injective and $\pi_2(N)\to \pi_2(M)$ is surjective.

  If $M$ is spin and the Rosenberg index of $N$ doesn't vanish, $0\ne \alpha(N)\in
  \rmK_*(C^*_\mathrm{max}\pi_1(N))$, then $M$ does not admit a Riemannian metric with
  positive scalar curvature.
\end{theorem}

The proof uses index theory of the Dirac operator, but not on $M$ itself. It
is a common theme that $\alpha(M)$ should be the most powerful index theoretic
obstruction to positive scalar curvature on $M$, formulated as a conjecture in
\cite{SchickICM}. However, for the codimension $2$ obstruction this remained
elusive. Some partial results in this direction have been carried out in
\cite{Engel} and \cite{NitscheSchickZeidler}, but they were of topological
nature and required the validity of the strong Novikov conjecture for
$\pi_1(M)$ to answer the question.

Higson, Schick, and Xie proved in \cite{HigsonSchickXie} a companion result:
the unexpected homotopy invariance of higher signatures of submanifolds of
codimension $2$ in the situation of Theorem \ref{theo:cd2o}, with the spin
condition replaced by orientability.

Recently, in \cite[Section 3.3]{Kubota} the first author proved the desired result about the Rosenberg index of $M$:
\begin{theorem}\label{theo:Kubota}
  In the situation of Theorem \ref{theo:cd2o}, the Rosenberg index of $M$
  itself is non-zero: $0\ne \alpha(M)\in \rmK_*(C^*_{\mathrm{max}}\pi_1(M))$.

More precisely, the situation determines a homomorphism (induced from a KK-element)
  \begin{equation*}
    \rho_{M,N} \colon \rmK_*(C^*_\mathrm{max}\pi_1(M))\to \rmK_{*-2}(C^*_\mathrm{max}\pi_1(N)),
  \end{equation*}
  mapping the Rosenberg index of $M$ to the one of $N$.
\end{theorem}

The proof uses a relative higher index setup, using more or less the interplay of $\pi_1(N\times S^1)\subset\pi_1(W)$ for the manifold $W:=M\setminus U(N)$ with boundary constructed from $M$ by removing a tubular neighborhood $U(N)\iso N\times D^2$ of $N$.

In this note, we give an account of the proof in \cite{Kubota}, together with a slight
simplification avoiding relative higher index theory. Moreover, we observe that this method also gives a simpler proof of the main result of \cite{HigsonSchickXie}, even of the following strengthening. Here we write $\sgn(M;\Nu_M)$ for the $C^*$-algebraic higher signatures, i.e.\  the Mishchenko-Fomenko index of the signature operator twisted with the Hilbert $C^*$-module bundle $\Nu_M :=\tilde{M} \times _{\pi_1(M)}C^*_{\mathrm{max}}\pi_1(N)$.

\begin{theorem}\label{theo:codim2sig}
  Let $M$ with submanifold $N$ be as in Theorem \ref{theo:cd2o}. If $M$ is oriented then the homomorphism $\rho_{M,N}$ of Theorem \ref{theo:Kubota} maps the $C^*$-algebraic higher signature of $M$ to twice the one of $N$. More generally, let $f\colon M_1\to M$ be a map between closed oriented manifolds of degree $1$. Assume $N$ is transversal to $f$ and set $N_1:=f^{-1}(N)\subset M_1$. Then
 \begin{equation}\label{eq:inddiff}
  \rho_{M,N}(\sgn(M_1; f^*\Nu_M)) = 2 \sgn(N_1; f^*\Nu_N) \in \rmK_*(C^*_\mathrm{max}\pi_1(N)).
 \end{equation}
In particular, if $f$ is a homotopy equivalence then $f^*\Nu_M=\Nu_{M_1}$ and by homotopy invariance of the $C^*$-algebraic signature $\sgn(M_1;\Nu_{M_1}) = \sgn(M,\Nu_M)$ we also get
   \begin{equation*}
2(\sgn(N_1;f^*\Nu_{N})-\sgn(N;\Nu_N)) =0,
\end{equation*}
recovering the main theorem of \cite{HigsonSchickXie}.
\end{theorem}

\subsubsection*{Acknowledgements}
 The first author was supported by RIKEN iTHEMS Program and JSPS KAKENHI
 Grants Numbers 19K14544 and JPMJCR19T2.

 The second author was supported by the DFG within the priority program
 ``Differential geometry at infinity''.

\section{Proof of the theorem}\label{sec:proof}

We set $\pi:=\pi_1(N)$, $\Gamma:=\pi_1(M)$, and $\Pi:=\pi\times
\integers$. Moreover, we set
$m:=\dim(M)$. Note that $n:=\dim(N)\equiv m\pmod 2$. Until the last section,
all group 
$C^*$-algebras are maximal group $C^*$-algebras and we simply write
e.g~$C^*\pi$ for the maximal group $C^*$-algebra of $\pi$.

The first step in the proof of the main theorem, used in all approaches, is a
well known product formula for the Mishchenko index of the Dirac, as well as
the signature operator of $N\times S^1$:
\begin{proposition}\label{prop:Kuenneth}
  Let $N$ be a closed manifold with fundamental group $\pi$ and identify
  $\pi_1(N\times S^1)$ with $\Pi$. The K\"unneth formula
  provides canonical homomorphisms
  \begin{equation*}
    \rmK_*(C^* \pi) \to \rmK_{*+1}(C^* \Pi) \xrightarrow{\beta}
    \rmK_*(C^* \pi)
  \end{equation*}
 (indeed induced by KK-elements)  with composition the identity.   If $N$ is spin these homomorphisms map the Mishchenko-Fomenko index of $N$ and of
  $N\times S^1$ to each other.

  If $N$ is oriented and $\cW$ is a bundle of finitely generated projective $C^*\pi $-modules on $N$, $\beta$ maps the analytic higher signature class $\sgn(N \times  S^1 ; \cW \boxtimes \cV_{S^1})$ (where $\cW \boxtimes \cV_{S^1}$ denotes the exterior tensor product bundle) to $2^\epsilon  \sgn(N; \cW) $ with 
    \[ \epsilon=
  \begin{cases}
    0; &m\equiv 0\pmod 2, \\
    1; & m\equiv 1\pmod 2.
  \end{cases}\] 
\end{proposition}
This is well known and follows from the principle ``boundary of Dirac is
Dirac'' (compare e.g.~\cite{Rosenberg}) and ``boundary of signature is
$2^\epsilon$ times signature'', this factor of $2$ for the signature operator
is explained e.g.~in \cite[2.13]{HigsonSchickXie}.

It therefore suffices to relate the Mishchenko-Fomenko index of Dirac and signature
operator on 
$N\times S^1$, considered as the boundary of the tubular neighborhood of $N$
in $M$ to the one of $M$.

For this, we make the following well known observations.
\begin{enumerate}
\item As a purely topological fact, we have a Mayer-Vietoris sequence in
  generalized homology for the
  decomposition $M=W\cup_{N\times S^1}N\times D^2$ where we identify a closed
  tubular neighborhood of $N$ in $M$ with $N\times D^2$ and let $W$ be the
  complement of its interior.

  We apply this to K-homology and get the boundary map
  \begin{equation*}
    \boundary_{MV}\colon \rmK_*(M)\to \rmK_{*-1}(N\times S^1).
  \end{equation*}
  A spin structure on $M$ defines a fundamental K-homology class $[M]\in
  \rmK_{m}(M)$, and
  \begin{equation*}
\boundary_{MV}([M])= [N\times S^1]\in \rmK_{m-1}(N\times S^1)
\end{equation*}
is the fundamental
  class of $N\times S^1$, another implementation of ``boundary of Dirac is
  Dirac''. For the signature operator of an orientation of $M$, we similarly
  get
  \begin{equation}\label{eq:sgn_kuen}
    \boundary_{MV}(f_*[M_1]_{\sgn}) = 2^{\epsilon+1}f_*[N_1 \times S^1]_{\sgn} \in \rmK_{m-1}(N \times S^1).
  \end{equation}
  Here $[M]_{\sgn}$ is the K-homology class of the signature operator and as
  before $\epsilon=
  \begin{cases}
    0; &m\equiv 0\pmod 2\\
    1; & m\equiv 1\pmod 2
  \end{cases}$.
\item The Mishchenko bundle
\[ \Nu_{N\times S^1}= \widetilde{N\times S^1}\times_{\Pi } C^* \Pi\]
of $N\times S^1$ defines a class in K-theory with coefficients in $C^*\Pi$, namely
\[ [\Nu_{N\times S^1}]\in \rmK^0(N\times S^1;C^* \Pi) := \rmK_0(C(N\times S^1)\tensor C^* \Pi).\]
Moreover, the Mishchenko-Fomenko index of the Dirac
  operator is simply obtained as the pairing of the fundamental K-homology
  class of $N\times S^1$ with this K-theory class of $N\times S^1$ as
  \begin{equation*}
    \alpha(N\times S^1) = \innerprod{[N\times S^1], [\Nu_{N\times S^1}]} \in
    \rmK_{m-1}(C^* \Pi ).
  \end{equation*}
  The corresponding statement holds for the signature and for $M$, in
  particular $\alpha(M)=\innerprod{[M],[\Nu_M]}\in \rmK_m(C^*\Gamma)$. 
\item The pairing between K-homology and K-theory is compatible with the
  Mayer-Vietoris boundary map:
  \begin{align*}
\innerprod{[N\times S^1], [\Nu_{N\times S^1}]}_{\rmK_*(C^*\Pi )}  =& \innerprod{\boundary_{MV}([M]), [\Nu_{N\times S^1}]}_{\rmK_*(C^*\Pi)} \\
=&  \innerprod{[M],\delta_{MV}[\Nu_{N\times  S^1}]}_{\rmK_*(C^* \Pi)} 
\end{align*}
  in $\rmK_m(C^* \Pi )$.
 Here $\delta_{MV}\colon \rmK^0(N\times S^1; C^* \Pi )\to \rmK^1(M; C^* \Pi )$ is the boundary map of the Mayer-Vietoris sequence in K-theory for the above decomposition of $M$, where we use fixed coefficients in $C^* \Pi$. 
\end{enumerate}

All of this is standard and relatively easy to derive. For the proof of the main theorem it therefore remains ``only'' to relate $  \innerprod{[M],[\Nu_M]}\in \rmK_m(C^* \Gamma)$ to $\innerprod{[M],\delta_{MV}[\Nu_{N\times S^1}]}\in \rmK_m(C^* \Pi )$, and for this the obvious strategy is to relate $[\Nu_M]\in \rmK^0(M;C^* \Gamma)$ to $\delta_{MV}[\Nu_{N\times S^1}]\in \rmK^1(M;
C^*  \Pi )$. This latter task, however, is not obvious at all and is achieved in \cite[Section 3.3]{Kubota}, stated in different terms there. 

We retrace these steps here, using our slightly simpler setup.
First, construct the standard Hilbert $C^*\Pi$-module $\sH_{C^*\Pi}:=\ell^2\bN \otimes C^*\Pi$ with its $C^*$-algebras of compact and bounded adjointable operators $\bK_{C^*\Pi}:=\bK_{C^*\Pi}(\sH_{C^*\Pi})$ and $\ \bB_{C^*\Pi}:=\bB_{C^*\Pi}(\sH_{C^*\Pi}) $, giving
rise to the short exact sequence of $C^*$-algebras
\begin{align*}
  0\to \bK_{C^*\Pi} \to \bB_{C^*\Pi}
  \to \cQ_{C^*\Pi}:=\bB_{C^*\Pi}/\bK_{C^*\Pi} \to 0. 
\end{align*}
The associated long exact sequence in K-theory yields the boundary isomorphism
(induced by a KK-element)
\begin{equation}\label{eq:boundaryiso}
\delta_\cQ \colon  \rmK_*(\cQ_{C^*\Pi}) \to \rmK_{*+1}(C^*\Pi)
\end{equation}
because $\rmK_*(\bB_{C^*\Pi})=0$ and by stability of K-theory we have
the canonical isomorphism $\rmK_*(C^*\Pi)\iso \rmK_*(\bK_{C^*\Pi})$.

The standard construction of the boundary homomorphism as in
\eqref{eq:boundaryiso} allows us here to explicitly find a representative of
the class in $\rmK_1(\cQ_{C^*\Pi})$ corresponding to $[\Nu_{N\times S^1}]$:
\begin{lemma}
  Consider the stabilized bundle $\Nu_{N\times S^1}\oplus
  \underline{\sH_{C^*\Pi}}$ of Hilbert $C^*\Pi$-modules. Its structure group, the unitary group $\cU(\bB_{C^*\Pi}(C^*\Pi\oplus \sH_{C^*\Pi}))$ with the norm topology, is contractible by the appropriate version of Kuiper's
  theorem. Therefore we have a trivialization (unique up to homotopy)
  $\underline{\sH_{C^*\Pi}}\to \Nu_{N\times S^1}\oplus
  \underline{\sH_{C^*\Pi}}$. Composed with the obvious projection $\Nu_{N\times
    S^1}\oplus \underline{\sH_{C^*\Pi}} \to \underline{\sH_{C^*\Pi}}$ we obtain a
norm continuous map $\Psi \colon N \times S^1  \to \bB_{C^*\Pi}$
  with values epimorphisms with finitely generated projective
  kernel. Therefore, composed with the projection to the Calkin algebra
  $\cQ_{C^*\Pi}$ the map takes values in
  unitaries and gives $p\circ \Psi \colon N\times S^1\to \cU(\cQ_{C^*\Pi})$, thus representing a class
  \begin{equation*}
    [p\circ\Psi]\in \rmK^1(N\times S^1; \cQ_{C^*\Pi}).
  \end{equation*}
 Its image under the map induced by $\delta_{\cQ}$ of \eqref{eq:boundaryiso} is
 precisely $[\Nu_{N\times S^1}] \in \rmK^0(N\times S^1;C^*\Pi)$. 
\end{lemma}
\begin{proof}
  There is a standard way to compute the boundary map: Lift $p\circ\Psi \colon N\times S^1\to \cU(\cQ_{C^*\Pi})$ to $\Psi$. The kernel is obviously isomorphic to the continuous section space $C(N \times S^1; \Nu_{N\times S^1})$ and the cokernel is trivial. Then the kernel represents the image under the boundary map, which corresponds to the class of the bundle $[\Nu_{N\times S^1}]$ in $\rmK^0(N\times S^1;C^*\Pi)$ in the bundle description of $\rmK^0$.
\end{proof}

Next, there is a standard description of the Mayer-Vietoris boundary map
\begin{equation}\label{eq:MV_cohom}
  \delta_{MV}\colon \rmK^1(N\times S^1;\cQ_{C^*\Pi})\to \rmK^0(M;\cQ_{C^*\Pi}):
\end{equation}

\begin{lemma}\label{lem:expl_MVim}
The class $\delta_{MV}([p\circ\Psi])\in \rmK^0(M; \cQ_{C^*\Pi})$ is represented by the Hilbert
$\cQ_{C^*\Pi}$-module bundle $V$ obtained by gluing the trivial free $\cQ_{C^*\Pi}$-module bundles of rank one over $W$ and over $N\times D^2$ along their common boundary using the isomorphism $[p\circ\Psi]\colon N\times S^1\to \cU(\cQ_{C^*\Pi})$. Here, we glue via left multiplication whereas the right module structure of the fibers is given by right multiplication.
\end{lemma}

Finally, taking the Mayer-Vietoris sequence is compatible with the K-theory
sequence induced from an extension of coefficient $C^*$-algebras, meaning
\begin{equation}
  \label{eq:MV_comp}
  \delta_{\cQ}^{-1}(\delta_{MV}[\Nu_{N\times S^1}] ) =
  \delta_{MV}(\delta_{\cQ}^{-1}[\Nu_{N\times S^1}]) = \delta_{MV}([p\circ\Psi]) = [V].
\end{equation}

Again, these are standard constructions. The main point now is the following
crucial theorem, which in a different form is the main idea of \cite[Section 3.3]{Kubota}.
\begin{theorem}\label{theo:Kubota_expl}
  There is a unitary representation $\rho\colon\Gamma\to \cU(\cQ_{C^*\Pi})$ where we
  consider 
  $\cQ_{C^*\Pi}$ as the $\cQ_{C^*\Pi}$-endomorphisms of the free
  $\cQ_{C^*\Pi}$-module of rank one such that the associated Hilbert
  $\cQ_{C^*\Pi}$-bundle $U$ (which is flat) is isomorphic to the bundle $V$ of
  Lemma \ref{lem:expl_MVim}.
\end{theorem}

With the other preparations, this is the heart of the proof of the main
theorems: 
\begin{corollary}\label{corol:main_corol}
  The class $\delta_{\cQ}^{-1}(\alpha(N\times S^1)) \in \rmK_m(\cQ_{C^*\Pi})$ is the
  image under $\rho_*\colon \rmK_m(C^*\Gamma)\to \rmK_m(\cQ_{C^*\Pi})$ of $\alpha(M)$
  where $\rho_*\colon C^*\Gamma\to \cQ_{C^*\Pi}$ is the homomorphism induced by
  the representation $\rho$ of Theorem \ref{theo:Kubota_expl} via the
  universal property of $C^*\Gamma$.

  Consequently, $\alpha(N)$ is the image of $\alpha(M)$ under the
  composition
  \begin{equation*}
    \rho_{M,N}:= \beta\circ\delta_{\cQ} \circ \rho_*\colon \rmK_*(C^*\Gamma)\to \rmK_{*-2}(C^*\pi)
  \end{equation*}
(with the K\"unneth map $\beta$ of Proposition \ref{prop:Kuenneth}).

 In the Situation of Theorem \ref{theo:codim2sig}, for the
  $C^*$-algebraic signature we get
  \begin{equation*}
	   \rho_{M,N}(\sgn(M_1;f^*\Nu_M)) = 2\sgn(N_1;f^*\Nu_N).
\end{equation*}

\end{corollary}
\begin{proof}
For any $\xi \in \rmK_0(M)$ we have
\begin{align*}
\innerprod{\xi,[U]} \stackrel{~\ref{theo:Kubota_expl}}{=} \innerprod{\xi ,[V]}
    \stackrel{~\ref{lem:expl_MVim}}{=} \innerprod{\xi ,\delta_{MV}\delta_{\cQ}^{-1}[\Nu_{N\times S^1}]}
    &= \innerprod{\partial_{MV}\xi,\delta_{\cQ}^{-1}[\Nu_{N\times S^1}]}  \\ &= \delta_{\cQ}^{-1}( \innerprod{ \partial _{MV} \xi, [\Nu_{N\times S^1}]}). 
 \end{align*} 
 Applying $\beta \circ \delta_\cQ$ to both sides, we get
 \begin{align} \beta \circ \delta_{\cQ} (\langle \xi, [U] \rangle ) = \beta (\innerprod{ \partial _{MV} \xi, [\Nu_{N\times S^1}]}) \label{eq:rho} \end{align}
 
In the case that $\xi$ is the Dirac fundamental class $[M]$, (\ref{eq:rho}) implies
\[ \rho_{M,N}(\alpha(M)) = \beta (\alpha(N \times S^1)) = \alpha(N)\] 
since $\rho_*(\alpha(M))=\innerprod{[M],[U]}\in \rmK_m(\cQ_{C^*\Pi})$ by \cite[Lemma 3.1]{HankeSchick}. 
Similarly, applying (\ref{eq:rho}) to the pushed signature class $\xi = f_*[M_1]_{\sgn}$ we get
 \begin{align*}
 \rho_{M,N}\sgn (M_1; f^*\cV_M) &= ( \beta \circ \delta_\cQ) (\langle f_*[M_1]_{\sgn}, [U] \rangle )\\
& \stackrel{~\ref{eq:rho}}{=} \beta (\langle \partial_{MV}f_*[M_1]_{\sgn}, [\cV_{N \times S^1}] \rangle )\\
 &\stackrel{~\ref{eq:sgn_kuen}}{=} 2^\epsilon \beta (\langle f_*[N_1 \times S^1]_{\sgn}, [\cV_{N \times S^1}] \rangle)  \\
&= 2^{\epsilon } \sgn (N_1 \times S^1 , f^*\cV_{N \times S^1}) \\
 & \stackrel{~\ref{prop:Kuenneth}}{=} 2 \sgn (N_1; f^*\cV_N),
\end{align*}
where $\epsilon$ is as in Theorem \ref{theo:codim2sig}. 
For the last equality, we remark that there is an isomorphism $f^*\cV_{N \times S^1} \cong f^* \cV_N \boxtimes \cV_{S^1}$. 
This is because $f|_{N \times S^1}$ is identified with $f|_N \times \id_{S^1}$ under the trivialization of the normal bundle of $N_1$ pulled back from that of the normal bundle of $N$.
\end{proof}

It remains to prove Theorem \ref{theo:Kubota_expl}. For this, following \cite[Section 3.3]{Kubota}
we construct the flat Hilbert $\cQ_{C^*\Pi}$-module bundle $U$ ad hoc, corresponding automatically to a representation $\rho$, and then check by hand that $U$ and $V$ are isomorphic Hilbert $\cQ_{C^*\Pi}$-module bundles.

For this, first we choose a basepoint $\ast$ in $N \times S^1 \subset M$ and identify the fundamental groups as $\Gamma = \pi_1(M, *)$, $\pi=\pi_1(N \times D^2,\ast)$ and $\Pi=\pi_1(N \times S^1, \ast)$. Let $q\colon (M_\pi,\ast)\to (M,\ast)$ be the covering projection with $\pi_1(M_\pi , \ast)=\pi\subset \Gamma$ and let $N\times D^2\subset M_\pi$ be the corresponding lift of the embedded $N\times D^2\subset M$, namely the connected component of $q^{-1}(N \times D^2)$ containing the basepoint. 
Define $W_\infty:= M_\pi\setminus(N\times D^2)$. 
It is a crucial consequence of the conditions on $\pi_2$, derived in \cite[Theorem 4.3]{HankePapeSchick}, that the inclusion of the boundary $N\times S^1\into W_\infty$ induces an injection $\Pi \to  \pi_1(W_\infty, \ast)$ with a left inverse $r \colon \pi_1(W_\infty, *) \to \Pi$ as in the diagram (\ref{eq:diag}). 
Let $W_\pi:=q^{-1}(W)\to W$ be the
restriction of the covering $M_\pi$ to $W$, a subset of $W_\infty$. Set
$G:=\pi_1(W, \ast)$ and $H:=\pi_1(W_\pi, \ast)$. The covering and inclusion maps give a
commutative diagram of spaces, inducing the one of fundamental groups
\begin{equation}\label{eq:diag}
\begin{split}
\xymatrix{
  N\times S^1\ar@{=}[r]\ar[d]^{\cap}&N \times S^1 \ar[r]^{\subset} \ar[d]^{\cap} & N \times D^2 \ar[d]^{\cap} \\
W_\pi \ar[r]^{\subset}_{i} \ar[d] &W_\infty \ar[r]^{\subset} & M_\pi \ar[d] \\
W \ar[rr]^{\subset }&&M
}\qquad 
\xymatrix{
\Pi\ar@{=}[r]\ar[d]& \Pi=\pi \times \bZ \ar[r] \ar@<0.5ex>[d] & \pi \ar@{=}[d] \\
H \ar[r]^{i_*\quad} \ar[d] &\pi_1(W_\infty ,\ast) \ar[r] \ar@<0.5ex>@{.>}[u]^r & \pi  \ar[d] \\
G \ar[rr] &&\Gamma .
}
\end{split}
\end{equation}
The horizontal maps of groups are surjective, the down vertical ones injective. The
inclusion of $N\times S^1$ into $W_\pi$ induces $\Pi \to H \to G$ (the first a split of $H\to \Pi $ of
\eqref{eq:diag}) and by the van Kampen theorem the normal subgroup $\Lambda :=
\langle \hspace{-0.3ex} \langle \bZ \rangle \hspace{-0.3ex} \rangle $ of $G$
generated by $\integers$ is the kernel of the epimorphism $G\onto
\Gamma$. Because every loop in $W$ which is null-homotopic in $M$ has to lift
to a \emph{loop} in $M_\pi$ and therefore also $W_\pi$, the inclusion $H\to G$
maps the kernel of $H\onto\pi$ isomorphically to this $\ker(G\to\Gamma)$.
Through the epimorphism $H\to \Pi$ we get an induced map
$H\to \cU(\bB_{C^*\Pi}(C^*\Pi))$ (acting by left multiplication). 
The associated Hilbert $C^*\Pi$-module bundle on $W_\pi$ extends to a flat bundle of $W_\infty$ (associated to the canonical map $\pi_1(W_\infty , \ast) \to \bB_{C^*\Pi}(C^*\Pi)$).

Inducing the representation $H \to \Pi \to \cU(C^*\Pi)$ up from $H$ to $G$ we
obtain the unitary representation of $G$ on the space of square-summable
sections of the induced $C^*\Pi$-bundle 
\[ \cX:= G \times _H C^*\Pi \cong \disjointunion _{gH \in G/H} C^*\Pi \]
on the discrete space $G/H \cong \Gamma / \Pi$, where the action is by left
multiplication. We denote it by
\[ \label{eq:rho_G}
  \rho _G \colon G\to \cU (\bB_{C^*\Pi}(\ell^2(G/H, \cX) )). \]
The above isomorphism $\cX \cong \disjointunion_{G/H} C^*\Pi$ means that $\ell^2(G/H, \cX)$ is the completion of an
(infinite) algebraic direct sum of free Hilbert $C^*\Pi$-modules of rank one,
on which $G$ acts via a combination of permutations and left
$\Pi$-multiplication. The corresponding flat Hilbert $C^*\Pi$-module bundle 
\[ \mathcal{H}_W:= \tilde{W} \times _{G} \ell^2(G/H, \cX)\]
on $W$ is the pushdown of the bundle $\tilde{W} \times _H C^*\Pi$ on $W_\pi$: the fiber over $x\in W$ is the completed direct sum of all the fibers in the inverse image of $x$ in $W_\pi$.

\begin{lemma}\label{lem:rep_on_Rand}
  Restricted to $\Pi$, the representation $\rho_\Pi:=\rho_G|_\Pi$ decomposes as a direct sum
  \begin{equation}\label{eq:rho_Pi}
  \rho_{\Pi}=\lambda_\Pi \oplus \rho_{\mathrm{rest}} \colon \Pi \to \cU(\underbrace{\bB_{C^*\Pi}(C^*\Pi)}_{C^*\Pi} \oplus
  \bB_{C^*\Pi}(\ell^2((G \setminus H)/H , \cX)))
  \end{equation}
  where the map $\lambda_\Pi$ to the first summand comes from left multiplication. Moreover, the map $\rho_{\mathrm{rest}}$ to the second summand factors through the
  projection $\pr_\pi \colon \Pi=\pi\times \integers\to \pi$, i.e.~is written as
  $\rho_\pi \circ \pr_\pi$ with 
  \[ \rho_\pi:= \rho_G|_{\pi} \colon \pi \to \cU(\bB_{C^*\Pi}(\ell^2((G \setminus H)/H , \cX))).\] 
Correspondingly, the restriction of $\mathcal{H}_W$ to $N\times S^1$ is the direct sum of flat bundles $\Nu_{N\times S^1}$ and 
  \[\mathcal{H}_{\mathrm{rest}}:=\widetilde{N \times S^1} \times _\Pi \ell^2((G \setminus H)/H , \cX), \]
 where  $\mathcal{H}_{\mathrm{rest}}$ extends to a flat bundle over $N\times D^2$.
\end{lemma}
\begin{proof}
Let $r$ denote the homomorphism as in (\ref{eq:diag}), let $i \colon W_\pi \to W_\infty$ denote the inclusion and let $\rho_H$ denote the restriction $\rho_G|_H$. 
The left multiplication by $H$ on $G$ preserves $H$ and $G \setminus H$, hence $\rho_H$ is decomposed as (\ref{eq:rho_Pi}). 
As $\Pi\to H \to \Pi$ is the identity, the resulting action on the summand $H\times_H C^*\Pi=C^*\Pi$ is given by left multiplication, i.e.~$\rho_H(h)|_{C^*\Pi} = (\lambda_\Pi \circ r \circ i_*)(h)$ for any $h \in H$.

We write $t$ for the generator of $\bZ \subset H \subset G$. Then
$\rho_\Pi(t)$ preserves each rank one direct summand $gC^*\Pi := \ell^2(\{
gH\}, gH \times _H C^*\Pi)$ of $\cH_{\mathrm{rest}}$ for any $gH \in
(G\setminus H)/H$, since $t \cdot gH = g \cdot g^{-1}tg H $ and $g^{-1}tg \in
H$. This observation also shows that $\rho_\Pi(t)|_{gC^*\Pi}$ is unitary
equivalent to $\rho_H(g^{-1}tg)|_{C^*\Pi}$. Now we show that $i_*(g^{-1}tg)=e$
for any $g \in G \setminus H$, which concludes that $\rho_\Pi(t)$ acts on
$\cH_{\mathrm{rest}}$ trivially, as $\rho_H$ factors through $i_*\colon H\to
\pi_1(W_\infty,*)$ by construction.

The element $g^{-1}tg\in H=\pi_1(W_\pi)$ is represented by the concatenation of the lift of the loop $g$ to a (non-closed, as $g\notin H$) path $\gamma$ in $W_\pi$ from the base point $\ast$, the corresponding lift of the loop $t\in S^1$ and the inverse of the path $\gamma$. We have to show that it is contractible in $W_\infty$. However, $\gamma$ ends in a different component of the inverse image of $N\times S^1$ under the covering projection $W_\pi\to W$. In $W_\infty$, this component is the boundary of a covering of $N\times D^2$ and therefore the lift of $t$ is contractible in $W_\infty$ and hence also $i_*(g^{-1}tg)=e \in \pi_1(W_\infty, \ast)$.
\end{proof}

\begin{definition}
  Define $\rho\colon \Gamma=G/\langle \hspace{-0.5ex}\langle\integers\rangle  \hspace{-0.5ex} \rangle \to \cU(\cQ_{C^*\Pi}) $ of Theorem \ref{theo:Kubota_expl}
  as induced by the composition
  \begin{equation}\label{eq:rhoGbar}
 \bar{\rho}_G \colon   G\xrightarrow{\rho_G} \bB_{C^*\Pi}(\ell^2(G/H, \cX)) \to
    \cQ_{C^*\Pi}(\ell^2(G/H, \cX) ) \cong \cQ_{C^*\Pi}, 
  \end{equation}
  using that the kernel $\Lambda$ of $G\to \Gamma$, normally generated by $\integers$,
  acts by Lemma \ref{lem:rep_on_Rand} as the identity on $\ell^2((G \setminus H)/H, \cX)$ and therefore as the identity in $\cQ_{C^*\Pi}$. 
\end{definition}

\begin{remark}\label{rem:bdl}
For the proof of Theorem \ref{theo:Kubota_expl}, we remark a relation between $\sH_{C^*\Pi}$-bundles and $\bB_{C^*\Pi}$-bundles. 
Let $\cH$ be a locally trivial bundle of Hilbert $C^*\Pi$-modules whose
typical fiber is $\sH_{C^*\Pi}$ and with structure group $\cU(\sH_{C^*\Pi})$
with norm topology. We write $\cU(\cH)$ for corresponding right principal $\cU(\sH_{C^*\Pi})$-bundle, with $\cU(\cH)_x = \cU(\sH_{C^*\Pi},\cH_x)$ for $x\in X$ (note that, in the usual convension of the product of endomorphisms, the products $(a,x ) \mapsto ax$ and $(x,b) \mapsto xb$ induce a left $\bB_{C^*\Pi}(\sH)$-action and a right $\bB_{C^*\Pi}(\sK)$-action on $ \bB_{C^*\Pi}(\sK, \sH)$). 
We have the associated bundles $\cB(\cH):= \cU(\cH) \times _{\cU (\sH_{C^*\Pi})}\bB_{C^*\Pi}$ and $\cQ(\cH):= \cU(\cH) \times _{\cU(\sH_{C^*\Pi})}\cQ_{C^*\Pi}$. Then the following are easily verified:
\begin{enumerate}
\item For the group $G=\pi_1(X)$ and its unitary representation $\rho \colon G \to \bB_{C^*\Pi}$, the associated bundle $\tilde{X} \times _G \bB_{C^*\Pi}$ is isomorphic to $\cB(\tilde{X} \times _G \sH_{C^*\Pi})$.
\item The fiber of the bundle $\cB(\cH)$ at $x \in X$ is $\bB_{C^*\Pi}(
  \sH_{C^*\Pi},\cH_x)$. Hence, a bounded bundle map $T \colon \cH \to \cK$ induces $
  \cB(\cH) \to \cB(\cK)$, fiberwise given by postcomposition. Similarly, $T$
  also induces $\cQ(\cH) \to \cQ(\cK)$.
\item If $\cH$ and $\cK$ are trivial, $T \colon X \times
  \sH_{C^*\Pi} \to X \times \sH_{C^*\Pi}$ is identified with a continuous
  function $X \to \bB_{C^*\Pi}$. The bundle map on $ \cB(X \times
  \sH_{C^*\Pi})  \cong X \times \bB_{C^*\Pi}$ (or $\cQ(X \times \sH_{C^*\Pi})
  \cong X \times \cQ_{C^*\Pi}$, respectively) induced from $T$ is the multiplication of $T$ from the left.
\end{enumerate}
\end{remark}

\begin{proof}[Proof of Theorem \ref{theo:Kubota_expl}]
We now prove that $V$ is isomorphic to $U$ by showing that the restrictions of $U$ to $W$ and to $N\times D^2$ both can be trivialized, and that the
change of trivialization over $N\times S^1$ is precisely the gluing map
which produces $V$.

For this, note that the restriction of $U$ to $W$ is the flat bundle associated to the representation $\overline{\rho_G}$ of \eqref{eq:rhoGbar}. As such, it is identified with $\cQ(\cH_W)$ by Remark \ref{rem:bdl} (1). 
We consider the flat bundle $\cQ(\cH_{\mathrm{rest}}) \cong \widetilde{N \times D^2} \times _{\rho_{\mathrm{rest}}} \cQ_{C^*\Pi}$ over $N \times D^2$ and glue them by the bundle map $\cQ(\cH_W)|_{N \times S^1} \to \cQ(\cH_{\mathrm{rest}})|_{N \times S^1}$ induced from the second projection
\[ w \colon \cH_{W}|_{N \times S^1} \cong \cV_{N \times S^1} \oplus \cH_{\mathrm{rest}} \to \cH_{\mathrm{rest}} \]
as is discussed in Remark \ref{rem:bdl} (2), to get the bundle $U'$. Since $w$
is a lift of the $\Pi$-invariant projection $\ell^2(G/H, \cX) \to
\ell^2((G\setminus H)/H, \cX)$ to the associated bundles, the flat
connection on $U|_W$ extends to $U'$ and its monodromy
representation coincides with $\bar{\rho}_G$. This means that $U \cong U'$.

By Kuiper's theorem, there are trivializations $v_W \colon W \times \sH_{C^*\Pi} \to \cH_W$ and $v_{N \times D^2} \colon N \times D^2 \times \sH_{C^*\Pi} \to \cH_{\mathrm{rest}}$. Then $U'$ is obtained by gluing two trivial $\cQ_{C^*\Pi}$-bundles along $N \times S^1$ by the bundle map induced from 
\[ v_{N \times D^2}^*wv_W \colon N \times S^1 \times \sH_{C^*\Pi} \to N \times S^1 \times \sH_{C^*\Pi}.\] 
This is a bundle map whose kernel
bundle is precisely $\Nu_{N\times S^1}$ and with trivial cokernels. By Remark \ref{rem:bdl} (3), this is precisely the gluing map for $V$, consequently $U \iso U' \iso V$ as Hilbert $\cQ_{C^*\Pi}$-module bundles. This finishes the proof of Theorem
\ref{theo:Kubota_expl} and therefore, in view of Corollary
\ref{corol:main_corol} also our two main results, Theorem \ref{theo:Kubota}
and Theorem \ref{theo:codim2sig}.
\end{proof}

\begin{remark}
  The Mayer-Vietoris boundary map $\delta_{MV}\colon \rmK^1(N\times
  S^1;\cQ_{C^*\Pi})\to \rmK^0(M;\cQ_{C^*\Pi})$ of \eqref{eq:MV_cohom} is not
  injective. In particular, $[V]=\delta_{MV}(\delta_{\cQ}^{-1}([\Nu_{N\times
    S^1}]-[\underline{C^*\Pi}]))$ is also associated to of the Mishchenko
  bundle over $N\times S^1$ and the trivial rank $1$ Hilbert $C^*\Pi$-module
  bundle over $N\times S^1$. The latter is used in some calculations of  \cite[Section 3.3]{Kubota}.
\end{remark}

\begin{proposition}
We have the following strengthening of Theorem \ref{theo:Kubota}: In the
situation of Theorem \ref{theo:Kubota}, the Rosenberg index $\alpha(M)$ is not
contained in the image of the map $\rmK_*(C^*_{\mathrm{max}}\pi_1(N)) \to
\rmK_*(C^*_{\mathrm{max}}\pi_1(M))$. This follows from Theorem \ref{theo:Kubota_expl}. Indeed, the composition $\rho \colon C^*\pi \to C^*\Gamma \to \cQ_{C^*\Pi}$ induces the zero map in $\rmK$-theory since the diagram
\[
\xymatrix{
C^*\pi \ar[r] \ar[d] & \bB_{C^*\Pi} \ar[d] \\
C^*\Gamma \ar[r] & \cQ_{C^*\Pi}
} 
 \]
 commutes and $\rmK_*(\bB_{C^*\Pi})=0$. Note that this strengthening is also a
 consequence of \cite{Kubota}*{Theorem 3.7}.
\end{proposition}

\end{document}